\documentclass[a4paper,10pt,reqno]{amsart}
\title{Upsilon-like concordance invariants from $\mathfrak{sl}_n$ knot cohomology}
\author{Lukas Lewark}
\address{Mathematisches Institut\\
Universit\"at Bern\\
Switzerland}
\email{lukas@lewark.de}
\urladdr{http://www.lewark.de/lukas/}
\author{Andrew Lobb} 
\address{Mathematical Sciences\\
Durham University\\
UK}
\email{andrew.lobb@durham.ac.uk}
\urladdr{http://www.maths.dur.ac.uk/users/andrew.lobb/}
\usepackage{thmtools,enumerate,graphicx,xcolor,enumitem}
\usepackage{etex,float}
\usepackage{ucs}
\usepackage{amsmath, amssymb, amscd, tabu, diagbox}
\usepackage{sseq}
\usepackage[latin1]{inputenc}
\usepackage{amsmath}
\usepackage{amssymb}
\usepackage{epsfig}
\usepackage{psfrag}

\makeatletter
\def\@tocline#1#2#3#4#5#6#7{\relax
  \ifnum #1>\c@tocdepth %
  \else
    \par \addpenalty\@secpenalty\addvspace{#2}%
    \begingroup \hyphenpenalty\@M
    \@ifempty{#4}{%
      \@tempdima\csname r@tocindent\number#1\endcsname\relax
    }{%
      \@tempdima#4\relax
    }%
    \parindent\z@ \leftskip#3\relax \advance\leftskip\@tempdima\relax
    \rightskip\@pnumwidth plus4em \parfillskip-\@pnumwidth
    #5\leavevmode\hskip-\@tempdima
      \ifcase #1
       \or\or \hskip 2em \or \hskip 2em \else \hskip 3em \fi%
      #6\nobreak\relax
    \hfill\hbox to\@pnumwidth{\@tocpagenum{#7}}\par
    \nobreak
    \endgroup
  \fi}
\makeatother

\usepackage[all]{xy}

\makeatletter
\newcommand{\LeftEqNo}{\let\veqno\@@leqno}
\makeatother
\newcommand{\khoca}{\texttt{khoca}}
\newlength{\lowerhalftmp}

\usepackage{microtype}

\usepackage{hyperref}
\usepackage[nameinlink]{cleveref}
\crefname{subsection}{subsection}{subsections}
\Crefname{subsection}{Subsection}{Subsections}
\definecolor{darkblue}{RGB}{0,0,96}
\definecolor{gray}{RGB}{127,127,127}
\definecolor{darkred}{RGB}{160,0,0}
\definecolor{lightyellow}{RGB}{255,255,128}
\hypersetup{colorlinks={true},linkcolor={black},citecolor={black},filecolor={black},urlcolor={black},pdfauthor={Lukas Lewark and Andrew Lobb},pdftitle={Upsilon-like concordance invariants from sl(n) knot cohomology}}

\Crefname{subsection}{Section}{Sections}
\Crefname{prop}{Proposition}{Propositions}
\crefname{equation}{}{}
\let\cref\Cref

\DeclareMathOperator{\ev}{ev}

\DeclareMathOperator{\im}{im}

\newcommand{\quax}{\hskip 0.75em plus 0.15em \ignorespaces}
\def\arxiv#1{\relax\ifhmode\unskip\quax\fi
    \href{http://arxiv.org/abs/#1}%
{\tt arXiv:\penalty -100\unskip#1}}    

\def\MR#1{\relax\ifhmode\unskip\quax\fi
    \href{http://www.ams.org/mathscinet-getitem?mr=#1}{MR#1}}
\def\xox#1{\!\!\!\!\csname xx#1\endcsname}

\newtheorem{theorem}{Theorem}[section]
\newtheorem{definition}[theorem]{Definition}
\newtheorem{lemma}[theorem]{Lemma}
\newtheorem{prop}[theorem]{Proposition}
\newtheorem{proposition}[theorem]{Proposition}

\theoremstyle{definition}
\newtheorem{remark}[theorem]{Remark}

\renewcommand{\epsilon}{\varepsilon}

\DeclareMathAlphabet{\mathpzc}{OT1}{pzc}{m}{it}
\usepackage{mathrsfs}

\newcommand{\Z}{\mathbb{Z}}
\newcommand{\CC}{\mathbb{C}}
\newcommand{\Q}{\mathbb{Q}}

\newcommand{\R}{\mathbb{R}}

\newcommand{\tensor}{\otimes}

\newcommand{\hash}{\#}

\newcommand{\SL}{\mathfrak{sl}}

\newcommand{\GG}{\mathcal{G}}
\newcommand{\gr}{{\rm Gr}}

\newcommand{\F}{\mathcal{F}}

\newcommand{\cC}{\mathcal{C}}

\newcommand{\upsln}{\gimel}

\newcommand{\CKR}{C_{\potential}}
\newcommand{\HKR}{H_{\potential}}

\graphicspath{{figs/}}
\newcommand{\potential}{\partial w}
\newcommand{\differential}{d}

\begin{document}

\keywords{Khovanov-Rozansky cohomology, Knot concordance, Knot Floer homology}
\subjclass[2010]{57M25}
\begin{abstract}
We construct smooth concordance invariants of knots $K$ which take the form of piecewise linear maps $\upsln_n(K) : [0,1] \rightarrow \R$ for $n \geq 2$.  These invariants arise from $\SL_n$ knot cohomology.  We verify some properties which are analogous to those of the invariant $\Upsilon$ (which arises from knot Floer homology), and some which differ.  We make some explicit computations and give some topological applications.

Further to this, we define a concordance invariant from equivariant $\SL_n$ knot cohomology which subsumes many known concordance invariants arising from quantum knot cohomologies.

\end{abstract}

\maketitle

\setcounter{tocdepth}{1}
\tableofcontents

\section{Introduction}
\label{sec:introduction}
\subsection{Where the invariants come from}
Given an oriented knot diagram $D$ of a knot $K$, a basepoint on $D$, and a choice of monic degree $n \geq 2$ polynomial (the \emph{potential}) $\potential \in \CC[x]$, the construction of the Khovanov-Rozansky $\SL_n$ knot cohomology gives a filtered cochain complex of finitely generated free $\CC[x]/\potential$-modules \cite{khr1,gornik,lobb1,wu3,krasnerEquivariant}:
\[ \cdots \subseteq \F^j C_{\potential}^i (D) \subseteq \F^{j+1} C_{\potential}^i (D) \subseteq \cdots {\rm ,}\]
\[ d : \F^j C_{\potential}^i (D) \longrightarrow \F^j C_{\potential}^{i+1} (D) {\rm ,} \,\, d^2 = 0 {\rm .}\]
This filtration $\F$ is known as the \emph{quantum} filtration.

In previous work \cite{lewarklobb}, the authors studied the associated graded vector space to the cohomology
\[ \gr^j H_{\potential}^i(D) := \frac{\F^{j} H_{\potential}^i (D)}{\F^{j-1} H_{\potential}^i (D)} \]
in the cases when $\potential$ is a product of distinct linear factors.  In these cases, the bigraded complex vector space $\gr^j H_{\potential}^i(D)$ is an invariant of $K$, it is of total dimension $n$, and supported in grading $i=0$.

The quantum gradings of the support of the cohomology give rise to lower bounds on the smooth $4$-ball genus of the knot, and it was a principal object of \cite{lewarklobb} to demonstrate that these bounds are heavily dependent on the choice of $\potential$, and display interesting behavior from various points of view.

For this paper, our starting point is somewhat different: we fix the potential to be $\potential = x^n - x^{n-1}$.  In the cases where $n \geq 3$ (which turn out to be the interesting cases) this potential is not a product of distinct linear factors.  The bigraded vector space $\gr^j H_{\potential}^i(D)$ is still an invariant of $K$, but it now has dimension equal to one more than the dimension of standard Khovanov-Rozansky $\SL_{n-1}$ cohomology \cite{rosewedrich}.  The copy of the $\SL_{n-1}$ cohomology arises from the root $x=0$ of $\potential$ which is of multiplicity $n-1$, while the extra copy of $\CC$ should be thought of as $\SL_1$ cohomology corresponding to the simple root $x=1$.

There is an easily described cocycle $\psi \in C_{\potential}^0 (D)$ generating this extra copy of $\CC$.  The minimum $j$ (suitably renormalized) such that
\[ [\psi] \in \F^{j} H_{\potential}^0 (D) \]
turns out to be a $\Q$-valued invariant of the smooth concordance class of $K$, and to provide a lower bound on the smooth $4$-ball genus of $K$.  It is not quite a concordance homomorphism to $\Q$ (as we shall see later in examples), but it is at least a \emph{quasi-homomorphism} (a homomorphism up to some bounded error).

We note here that this value of $j$ is characterized by being the minimal value of $j$ such there exists a $\psi' \in \F^{j} C_{\potential}^0 (D)$ which is \emph{cohomologous} to $\psi$.  This minimal value may well not be attained by $\psi$ itself, and in fact the quantum filtration grading of the cocycle $\psi$ corresponds to the `slice-Bennequin' bound on the smooth $4$-ball genus of $K$ arising from the diagram $D$.

So far this story has concerned the quantum filtration; now we introduce another filtration on $C_{\potential}^* (D)$ preserved by the differential, which we shall call the $x$\emph{-filtration} (previously used in the case $\potential = x^n$ in \cite{lew2}).  Given a cochain $c \in C_{\potential}^i (D)$ this is the filtration that simply counts the maximal power $k \leq n - 1$ of $x$ such that $c = x^k c'$ for some $c' \in C_{\potential}^i (D)$.  In other words it is the filtration
\[ \{0\} \subseteq x^{n-1} C_{\potential}^i(D) \subseteq x^{n-2} C_{\potential}^i(D) \subseteq \cdots \subseteq x C_{\potential}^i(D) \subseteq C_{\potential}^i(D) {\rm .}\]

The cocycle $\psi$ mentioned above is quite uninteresting when looked at from the point of view of the $x$-filtration. In fact, $\psi$ is a $1$-eigenvector for the action of $x$ so certainly we have
\[ \psi = x^{n-1} \psi \in x^{n-1} C_{\potential}^0(D) {\rm .}\]
However, there is the possibility that $\psi$ may be \emph{cohomologous} with elements which are both of lower quantum filtration and of lower $x$-filtration.  Indeed we shall give examples where this is the case, and it is the existence of such examples that makes our construction non-trivial.

Now that we have two filtrations on the cochain complex, we can blend them in a similar way to the blending of the algebraic and Alexander filtrations in knot Floer homology.  As in knot Floer homology, where the blending gives the invariant~$\Upsilon$, we get a piecewise linear map on an interval
\[ \upsln_n(D) : [0,1] \rightarrow \R { \rm .} \]

\subsection{A few words about equivariant cohomology}

Before turning to the comparison with invariants arising from Floer homology, we consider the most general setting in which concordance invariants arise from $\SL_n$ knot cohomology.

For a given knot diagram $D$ of a knot $K$, the equivariant cohomology of $K$ is the cohomology of a graded cochain complex $C_{U(n)}(D)$ of free modules over a multivariable polynomial ring \cite{krasnerEquivariant}.  By specializing these variables to take values in the complex numbers one obtains all of the filtered $\SL_n$ Khovanov-Rozansky cochain complexes $\CKR(D)$.  It follows that the equivariant cohomology subsumes all of the known information contained in Khovanov-Rozansky cohomology about the slice genus and concordance class of $K$, including that in $\upsln_n(K)$.  The $\Upsilon$-like properties and computability of $\upsln_n(K)$ nevertheless give it advantages over the full equivariant cohomology.

We pursue the equivariant viewpoint in the final section of the paper, extracting a concordance invariant directly.  This takes the form of a particular indecomposable summand $S_n(K)$ of $C_{U(n)}(D)$ up to isomorphism.  In the case of a knot $K$ for which this summand is a shifted free module of rank $1$,
the slice genus information that we know how to extract from $\SL_n$ cohomology only depends on the shift, and consequently $\upsln_n(K)$, for example, is linear for all $n \geq 2$.

\subsection{Some Floer homological invariants}
We now provide context for our main results by briefly discussing $\Upsilon$ and some other invariants arising from Floer homology.

Ozsv\'ath-Szab\'o \cite{osz10} and Rasmussen \cite{ras} defined the invariant $\tau$, which takes integer values on knots in the $3$-sphere.  This (or more precisely its negative) was the first example of a \emph{slice-torus} invariant.

\begin{definition}[{cf. \cite{livingston,lew2}}]
	\label{defn:slicetorus}
	Let $\nu: \cC \rightarrow \R$ be a homomorphism from the smooth concordance group of oriented knots to the reals.  We say that $\nu$ is a \emph{slice-torus} invariant if
	
	\begin{enumerate}
		\item $g_*(K) \geq |\nu(K)|$ for all oriented knots $K$, where we write $g_*(K)$ for the smooth slice genus of $K$.
		\item $\nu(T(p,q)) = \frac{-(p-1)(q-1)}{2}$ for $T(p,q)$ the $(p,q)$-torus knot.
	\end{enumerate}
\end{definition}

The second example -- the $s$ invariant (or more precisely $-s/2$) -- was due to Rasmussen \cite{ras3} and had a purely combinatorial definition in terms of the Lee perturbation \cite{lee} of Khovanov cohomology.  The reason for our normalization convention in \cref{defn:slicetorus} is that there is a slew of such invariants (see for example \cite{wu3}, \cite{lobb1}, \cite{lewarklobb}) arising from $\SL_n$ Khovanov-Rozansky cohomology, which (in the original definition \cite{khr1}) is supported in negative quantum gradings for non-trivial positive knots.

Slice-torus invariants are good, for example, for finding free summands of the knot concordance group, and sets of linearly independent slice-torus invariants are even more useful from this point of view.  A weakness of slice-torus invariants is that they all agree on quasi-positive and homogeneous knots \cite{kawamura, lobb5, lew2}. The class of homogeneous knots includes all alternating knots. For an alternating knot~$K$, every slice-torus invariant $\nu$ satisfies $\nu(K) = -\frac{\sigma(K)}{2}$, where $\sigma$ is the classical knot signature.  Therefore slice-torus invariants necessarily miss some of the information contained in the smooth concordance class of a knot.

Let us turn next to the knot invariant $\Upsilon$ defined by Ozsv\'ath-Stipsicz-Szab\'o \cite{ossz} and interpreted in an excellent survey article by Livingston \cite{liv}.  This takes the form of a piecewise linear map
$ \Upsilon(K) : [0,1] \rightarrow \R. $
The actual domain of definition of $\Upsilon(K)$ is the interval $[0,2]$, but we allow ourselves to consider this restriction since $\Upsilon(K)$ satisfies $\Upsilon(K)(1+t) = \Upsilon(K)(1-t)$.

We now collect some facts about $\Upsilon$ under the heading of the following theorem.
\begin{theorem}[\cite{ossz}]
	\label{thm:upsilonfactoids}
        \begin{enumerate}[label=(\arabic*)]
                \item $\Upsilon(K)$ is a smooth concordance invariant of $K$.
		\item $\Upsilon$ is a homomorphism from the smooth concordance group of knots to the group of piecewise linear functions on the interval.
		\item $\vert \Upsilon(K)(t) \vert \leq t g_*(K)$ for all $t$, where we write $g_*(K)$ for the smooth $4$-ball genus of $K$. \label{upsilon:3}
		\item For small $t$, $\Upsilon(K)(t) = - \tau(K) t$ (we may write this as $\Upsilon(K)'(0) = -\tau(K)$ where the right-hand derivative is understood).
		\item For quasi-alternating knots, we have $\Upsilon(K)(t) = -\tau(K)t = - \frac{\sigma(K)}{2} t$ for all $t \in [0,1]$.
	\end{enumerate}
\end{theorem}

We note further that the concordance homomorphism to $\R$ given by evaluating at the right-hand endpoint $K \mapsto \upsilon(K) : = \Upsilon(K)(1)$ is interesting because it gives rise to the lower bound $|\upsilon(K) - \frac{\sigma(K)}{2}|\leq\gamma_4(K)$ on the smooth non-orientable $4$-ball genus $\gamma_4$ \cite{ossz2}. It can be shown that $\upsilon$ always takes values in the integers.

Related to $\upsilon(K)$ is an invariant $\varphi(K)$ defined by Golla and Marengon \cite{goma} in terms of earlier invariants defined by Rasmussen \cite{ras5} and studied by Ni and Wu~\cite{niwu}.  This invariant $\varphi(K)$ gives rise to a similar lower bound, namely $\frac{\sigma(K)}{2} - \varphi(K) \leq \gamma_4(K)$.  In contrast to $\upsilon$, however, $\varphi$ takes values only in the \emph{non-negative} integers, so (given that it is not identically zero) it cannot be a concordance homomorphism.  However, it is at least subadditive with respect to the connected sum of knots, which we denote by $\#$:
\[ \varphi ( K_1 \# K_2 ) \leq \varphi(K_1) + \varphi(K_2) {\rm .} \]
We note that the related bound $\frac{\sigma(K)}{2} - \varphi(K)$ for $\gamma_4(K)$ is superadditive.

\subsection{Main results}
We begin by running down the list of properties of $\Upsilon$ given in \cref{thm:upsilonfactoids}, and seeing where those of $\upsln_n$ agree or differ.  Firstly, there is the question of concordance invariance.

\begin{theorem}
	\label{thm:upsln_concordance_invariant}
	The map $\upsln_n(D)$ is a knot invariant so we may write it as $\upsln_n(K)$.  Furthermore, $\upsln_n(K)$ only depends on the smooth concordance class of $K$.
\end{theorem}

Next we observe that $\upsln_n$ is almost a concordance homomorphism.

\begin{theorem}
	\label{thm:quasi-homomorphism}
	We have that $\upsln_n$ is a \emph{quasi-homomorphism} from the smooth concordance group of knots to the group of piecewise linear functions on the interval.  More precisely, if we write $K_1 \# K_2$ for the connect sum of the knots $K_1$ and $K_2$ then we have
	\[ |\upsln_n(K_1 \# K_2)(t) - \upsln_n(K_1)(t) - \upsln_n(K_2)(t)| \leq 2t \]
	for all $t \in [0,1]$.
\end{theorem}

This can be regarded as either a strength or a weakness of $\upsln_n$.  The property of being a homomorphism is restrictive, although it can be useful for some applications.  We can be more specific about the failure to be a concordance homomorphism.  
Firstly, note that there is no failure near $0$.

\begin{theorem}
	\label{thm:derivativeatzero}
	The right-handed derivative $\upsln_n (K)'(0)$  is a slice-torus concordance homomorphism in the sense of \cref{defn:slicetorus}.
\end{theorem}

This result puts us in line with the derivative at $0$ of $\Upsilon(K)$ as mentioned in \cref{thm:upsilonfactoids}.  At $1$, on the other hand, we have the property of superadditivity

\begin{theorem}
	\label{thm:superadditiveat1}
	We have that
	\[ \upsln_n(K_1 \# K_2)(1) \geq \upsln_n(K_1)(1) + \upsln_n(K_2)(1) {\rm .} \]
\end{theorem}

Both $\upsln_n(K)'(0)$ and $\upsln_n(K)(1)$ are new knot invariants in their own right. They are similar to concordance invariants considered in \cite{lewarklobb} coming from reduced and unreduced cohomology, respectively.
However, all previously defined $\SL_n$ concordance invariants stem from cohomology with a separable potential, in contrast to the new invariants from $\SL_n$ cohomology with potential $x^n - x^{n-1}$.

\Cref{thm:superadditiveat1}
can be compared with the superadditive property of the bound on $\gamma_4$ arising from $\varphi$ discussed in the previous subsection.  In fact, let us now turn to the question of bounds on the smooth $4$-ball genus.

\begin{theorem}
	\label{thm:slicegenusbounds}
	For any $0 < t \leq 1$ we have that
	\[ g_*(K) \geq \left\vert \frac{\upsln_n(K)(t)}{t} \right\vert {\rm .} \]
\end{theorem}

This is in direct analogy with the property \ref{upsilon:3} of $\Upsilon$ given in \cref{thm:upsilonfactoids} and is implied by the next, more general, proposition.

\begin{proposition}
	\label{prop:genus_cobordism_bounds}
	For two knots $K_0$, $K_1$, we write $g_*(K_0, K_1)$ for the minimal genus of a knot cobordism from $K_0$ to $K_1$.  Then we have that
	\[ g_*(K_0,K_1) \geq \frac{1}{t}\left\vert \upsln_n(K_0)(t) - \upsln_n(K_1)(t) \right\vert {\rm .} \]
\end{proposition}

It shall turn out that it really is necessary to upgrade from Khovanov 
cohomology to $\SL_n$ Khovanov-Rozansky cohomology in order to obtain a 
non-trivial invariant.  Khovanov cohomology is equivalent to the case $n=2$ of 
Khovanov-Rozansky cohomology, and the well-known slice-torus invariant arising 
from Khovanov cohomology is just a scalar multiple of Rasmussen's invariant 
$s(K)$.

\begin{proposition}
	\label{prop:sl2isboring}
	We have
	\[ \upsln_2(K)(t) = \frac{-s(K)}{2}t {\rm .} \]
\end{proposition}

Of course, this implies for quasi-alternating knots $K$ that we have
\[ \upsln_2(K)(t) = \frac{- \sigma(K)}{2}t \]
where $\sigma$ is the classical knot signature.  It is a weakness of $\Upsilon$ that it contains no more information than $\sigma$ when applied to quasi-alternating knots, \cref{prop:sl2isboring} shows that $\upsln_2$ suffers from a similar weakness.  We shall see too that $\upsln_n$ is in general uninteresting for some classes of knots (in particular torus knots) for which $\Upsilon$ can be interesting.

\begin{proposition}
	\label{prop:boring_for_sqzd_knots}
	For any knot $K$ which is either quasi-positive, quasi-negative, or homogeneous, $\upsln_n(K)$ is linear.
\end{proposition}

On the other hand, we find interesting (in other words non-linear) values even of $\upsln_3$ on quasi-alternating knots.  This is in contrast to $\Upsilon$ or to $\varphi$, neither of which can distinguish a quasi-alternating knot from the $(2,2n+1)$ torus knot of the same signature.  As an example of the power of $\upsln_n$, we have the following result.

\begin{proposition}
	\label{prop:existence_of_knot_beating_upsilon_etcetera}
	There exists a knot $K$ on which $\varphi$, $\Upsilon$, and all known slice-torus invariants are trivial, but which cannot be slice since, for example, $\upsln_3(K) \not= 0$.
\end{proposition}

\begin{figure}[t]
        \centering
	\includegraphics[scale=0.08]{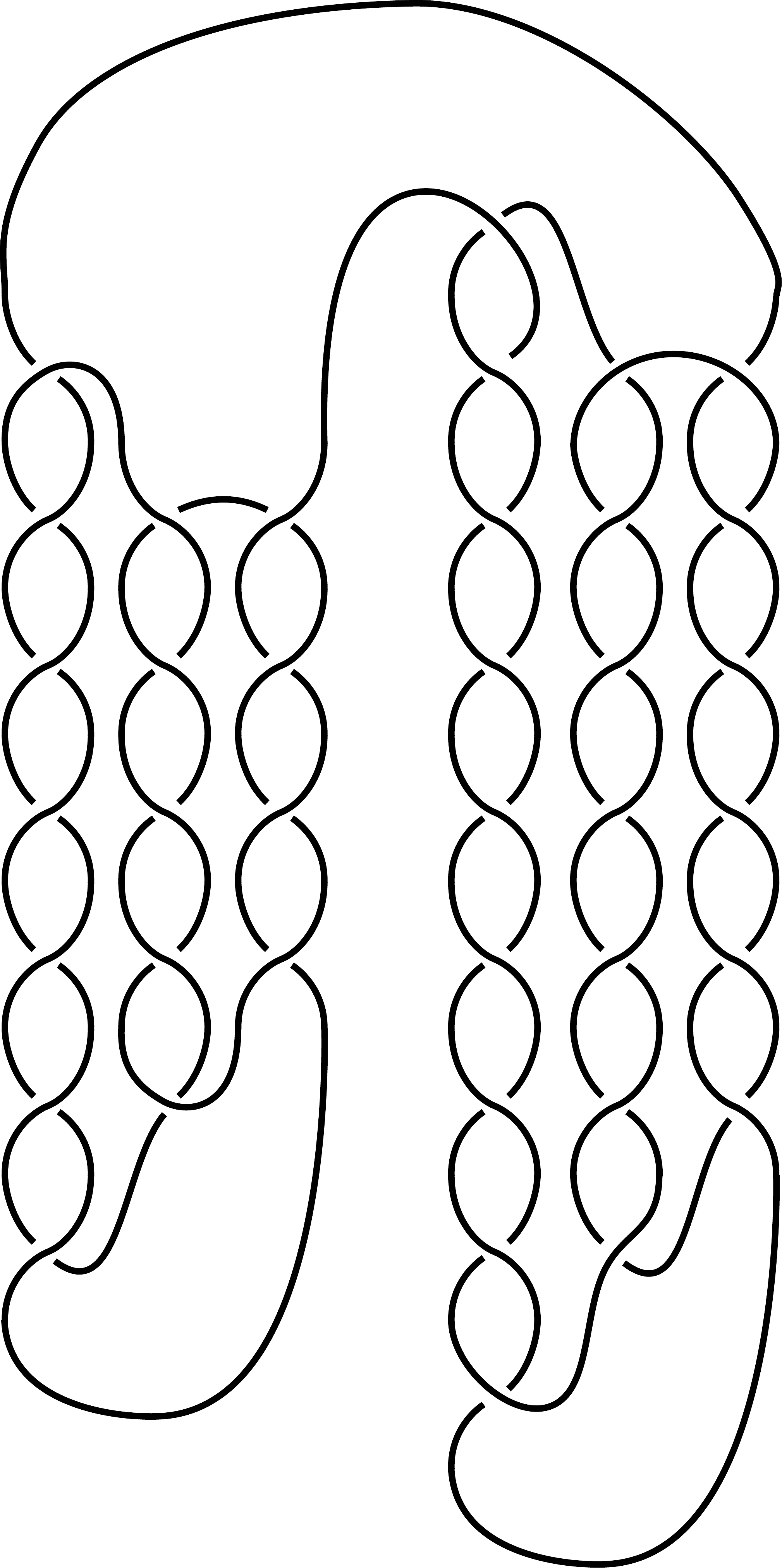}
	\caption{The connect sum of the pretzel knots $P(7,-5,4)$ and $P(-9,7,-6)$.}
	\label{fig:pretzel_example}
\end{figure}

A knot $K$ satisfying the properties of \Cref{prop:existence_of_knot_beating_upsilon_etcetera} can in fact be given explicitly, and an example is given in \cref{fig:pretzel_example}.  We consider this knot $K$ in detail in \Cref{sec:outlook} where we also show that the properties of $\upsln_n$ given above can be used to see that $K$ is of infinite order in the concordance group.  In fact, for example, we can also deduce the following result not obtainable by known invariants.

\begin{proposition}
	\label{prop:inf_order_in_QA_mod_A}
	We write $\langle {\rm QA} \rangle$ (respectively $\langle {\rm A} \rangle$) for the subgroup of the concordance group generated by quasi-alternating (respectively alternating) knots.  The knot given in \cref{fig:pretzel_example} is of infinite order in the group $\langle {\rm QA} \rangle/\langle {\rm A} \rangle$.
\end{proposition}

In \Cref{sec:outlook} we further consider the question of the independence of $\upsln_n$ and more classical concordance invariants such as generalized signatures.

\subsection{Discussion}
\label{subsec:discussion}
The slew of concordance invariants arising from quantum $\SL_n$ knot cohomology seems largely independent of those arising from Floer homology or gauge theory.  For example, $\upsln_n$ does not see any information beyond the slice genus for torus knots, while $\Upsilon$ does, but on the other hand $\upsln_n$ is found to be interesting for quasi-alternating knots while $\Upsilon$ must be `standard'.

Whether this independence can be pushed so far that one can find a knot $K$ for which $\upsln_n(K)$ is non-zero for some $n$ and all other known sliceness obstructions vanish, is perhaps less interesting than finding some new topological applications of $\upsln_n$ and related quantum invariants.  It is not obvious, for example, that $\upsln_n$ is insensitive to torsion elements of the concordance group (this is also non-obvious for the unreduced concordance invariants given in \cite{lewarklobb}).

The known exception to the orthogonality of quantum and Floer is Rasmussen's invariant~$s$, defined using Lee's perturbation of Khovanov cohomology.  Kronheimer-Mrowka \cite{km8} showed that $s$ is equal to a concordance homomorphism arising from $SU(2)$ instanton knot Floer homology.  One should then ask whether there is more concordance information than Rasmussen's invariant contained in Khovanov cohomology over the complex numbers (note that the Rasmussen invariant can be defined over any coefficient field using the Bar-Natan potential \cite{bncob}, and its value in general depends on the coefficient field \cite{LipSarKhov}).  Knot Floer homology, which is intimately connected with Khovanov cohomology, seems to admit many refined invariants, could the same be true of Khovanov cohomology?

On the other hand, $\SL_n$ knot cohomologies when $n \geq 3$ already give orthogonal concordance information to that arising from Floer homology.  In the case $n=2$ the failure of Khovanov cohomology to do the same is, roughly speaking, due to the existence of an unoriented skein exact sequence also often present in Floer homology theories.  But such skein exact sequences should not be present for instanton homologies with gauge group $SU(n)$ for $n \geq 3$.  This suggests that higher index Floer homologies should see much more of the concordance group than is seen by those most often currently studied.

\bigskip\paragraph{\emph{Acknowledgments:}}
The authors thank BIRS
and the Isaac Newton Institute for Mathematical Sciences (EPSRC grant~EP/K032208/1)
for support and hospitality during the programs \emph{Synchronizing Smooth and Topological 4-Manifolds}
and \emph{Homology theories in low dimensional topology}, respectively,
where work on this paper was undertaken.
The first author gratefully acknowledges support by the SNSF grant~159208.
The authors thank the referees for their careful reading and suggestions.

\section{Definitions and conventions}
\label{sec:defs}
In this section we lay out conventions for defining the invariant $\upsln_n(K) : [0,1] \rightarrow \R$.  We shall choose these definitions so that
\begin{itemize}
	\item $\upsln_n(K)$ is piecewise linear,
	\item $\upsln_n(K)(0) = 0$ for any knot $K$,
	\item $\upsln_n(U) = 0$ for $U$ the unknot,
	\item $\upsln_n(T_{-2,3}) (t) = t$ for $T_{-2,3}$ the left-handed trefoil.
\end{itemize}

In what follows, $D$ will denote a knot diagram with a basepoint and the potential is $\potential = x^n -x^{n-1}$.
  We write $\CKR(D)$ for the $\SL_n$ Khovanov-Rozansky cochain complex of free finitely-generated $(\CC[x]/\potential)$-modules arising from $D$.  We write
\[
\cdots \subseteq \F^j C_{\potential}^i (D) \subseteq \F^{j+1} C_{\potential}^i (D) \subseteq \cdots
\]
for the quantum filtration and
\[
\cdots \subseteq x^k C_{\potential}^i (D) \subseteq x^{k-1} C_{\potential}^i (D) \subseteq \cdots
\]
for the $x$-filtration. Both of these filtrations are preserved by the differential.
Let us point out that the quantum filtration is \emph{increasing} (the higher the index, the bigger the module),
whereas the $x$-filtration is \emph{decreasing} (the higher the index, the smaller the module).
So the quantum filtration grading of a non-zero cochain $c\in \CKR(D)$ is the minimum~$j$ such that
$c \in \F^j \CKR(D)$, while the $x$-filtration grading of $c$ is the maximum $k \leq n - 1$ such that $c\in x^k \CKR(D)$.

For all $t \in [0,1]$ we describe an increasing filtration with index set $\R$ of $\CKR(D)$ preserved by the differential.
We write this filtration as
\[ \GG_t^{\ell_1} C_{\potential}^i (D) \subseteq \GG_t^{\ell_2} C_{\potential}^i (D) \]
for all $\ell_1, \ell_2 \in \R$, $\ell_1 \leq \ell_2$.

\begin{definition}
	\label{defn:slopingfiltration}
	Writing $\Sigma$ to denote a sum of vector spaces, we define
	\[ \GG_t^{\ell} C_{\potential}^i (D) = \sum_{\substack{\ell \geq t(k+j) - k , \\ k \leq n-1}} (F^j\CKR^i(D) \cap x^k \CKR^i(D)) {\rm .} \]
\end{definition}
Let the $\GG_t$-filtration grading of a non-zero cochain $c\in \CKR^i(D)$ be the minimum~$\ell$ such that~$c \in \GG_t^{\ell}\CKR^i(D)$.

This filtration has another, more graphical, interpretation, which is useful for calculation and visualization.  Suppose that we draw the line of slope $\frac{t}{1-t}$ through the point $(\ell, -\ell)$.  Then if a point $(j,k) \in \mathbb{Z}\times \{0,\ldots, n-1\}$ is above or to the left of this line, it means exactly that
\[ \F^j \CKR^i (D) \cap x^k \CKR^i(D) \subseteq \GG_t^{\ell} \CKR^i (D) {\rm .} \]

There is a cocycle $\psi(D) \in \CKR^0(D)$ essentially first described by Gornik \cite{gornik} (although he considered rather the potential $x^n -1$) representing a non-zero cohomology class, see \cref{defn:psi}.

\begin{definition}
	\label{defn:almostgimel}
	We define $\gamma(D): [0,1] \rightarrow \R$ by
	\[ \gamma(D)(t) = \min \{ \ell : [\psi(D)] \in {\rm im}(H^0(\GG_t^{\ell} \CKR (D)) \rightarrow \HKR^0(D) ) \} {\rm ,}\]
	where the map on cohomologies is induced by inclusion.
\end{definition}
Equivalently, $\gamma(D)(t)$ is the filtration grading of the cohomology class $[\psi(D)] \in \HKR^0(D)$
with respect to the filtration induced on cohomology by $\GG_t$.

\begin{definition}
	\label{defn:gimel}
	We define
		\[ \upsln_n(D)(t) = \frac{1}{2(n-1)} [\gamma(D)(t) - \gamma(U)(t)] { \rm ,}\]
	where $U$ is the zero-crossing diagram of the unknot.
\end{definition}
Explicitly, $\gamma(U)(t) = (n - 1)(2t - 1)$, as one computes from \cref{defn:slopingfiltration}.

\section{A first example}
\label{sec:example}
As a first example of a knot with interesting (in other words, non-linear) $\upsln_n$, let us compute $\upsln_n(K)$ for all $n\in\{3,\ldots,10\}$ and $K=P(2,-3,7)$, a pretzel knot (DT-name~$12n_{235}$).
To lighten notation, we will mostly drop `$(K)$' in this section, writing $\upsln_n$ for $\upsln_n(K)$.  We shall give the calculation for general $n\in\{3,\ldots,10\}$, but the reader would do well to look at \Cref{fig:dotdiagram} for concreteness, where the bones of the calculation for $\upsln_5$ are displayed.

The starting point of the calculation shall be the equivariant complex of a diagram $D$ of
$K$ as defined by Krasner \cite{krasnerEquivariant}.
It is a complex, denoted by $C_{U(n)}(D)$, of graded free modules over the ring
\[
R_n = \mathbb{C}[x, a_0, \ldots, a_{n-1}] / (a_0 + a_1x + \ldots + a_{n-1}x^{n-1} + x^n).
\]
Up to homotopy equivalence, that complex is a knot invariant,
and it specializes to $C_{\potential}(D)$ when the formal variables $a_i$ are replaced by the coefficients of $\potential$.

\begin{figure}[t]
\def\svgwidth{350pt}
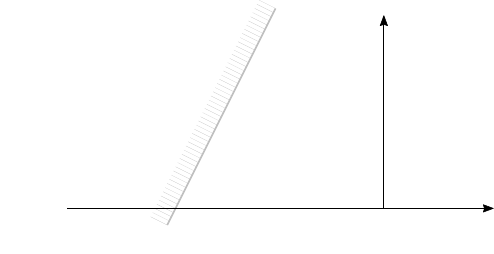
\caption{The filtered complex $S$ in cohomological grading $-1$ and $0$, special case $n = 5$,
as complex of complex vector spaces, i.e.\ forgetting the $x$-action.
Dots represent copies of $\mathbb{C}$
and arrows non-trivial differentials. The colored dots indicate the support of representatives of $[\psi]$.
Dots on or to the left of the gray line of slope $1$ lie in 
$\mathcal{G}^{-5}_{1/2}S$.}
\label{fig:dotdiagram}
\end{figure}

We will use $C_{U(n)}$ rather \emph{ad hoc}, and refer the reader to \cref{sec:equi} for a more conceptual and detailed treatment.
For a suitable diagram $D$ of $K$,
we computed $C_{U(n)}(D)$ for $n\in\{3,\ldots,10\}$ with our program \khoca{} \cite{khoca},
and found that for each of those values of $n$, $C_{U(n)}(D)$ is homotopy equivalent
to a sum of five simpler complexes, only one of which is
non-trivial in cohomological grading~$0$. It follows that this summand on its own determines $\upsln_n(K)$,
and hence we give it our exclusive attention.
The summand has the following form:
\[
\xymatrix@R=1.5ex@C=4em{
                          & t^0 q^{-2n}R_n \ar@{}[dd]|{\oplus} & \\
t^{-1}q^{-2}R_n \ar[ru]^{\potential'} \ar[rd]_{\potential''} \\
 & t^0q^{2-2n}R_n
}
\]
Here, $\potential'$ and $\potential''$ signify the first and second partial derivative with respect to $x$ of $\potential = x^n + a_{n-1}x^{n-1} + \ldots + a_0$, and $t^iq^jR_n$ denotes a free $R_n$-module of rank one in $(t,q)$-grading $(i,j)$.

Since $C_{U(n)}(D)$ exhibits this structure for all $n\in\{3,\ldots,10\}$,
it is a reasonable conjecture that it does in fact for all $n\geq 3$.
The calculations that follow are valid for all $n\geq 3$, and so they
 determine $\upsln_n$ for all $n\geq 3$ provided the conjecture holds.

One may plug in $a_{n-1} = -1, a_{n-2} = \ldots = a_0 = 0$ to obtain the relevant summand $S$ of a filtered complex of free $\CC[x]/(x^n - x^{n-1})$-modules homotopy equivalent to $C_{x^n-x^{n-1}}(D)$. 
Note that $S^0$ is of rank two, and so we will write the cochains in cohomological grading $0$ as vectors with two entries.
The complex $S$ has the following differential~$\differential^{-1}$:\\
\begin{align*}
\differential^{-1}(1) = & \begin{pmatrix}
 nx^{n-1} - (n-1)x^{n-2} \\
 n(n-1)x^{n-2} - (n-1)(n-2)x^{n-3}
\end{pmatrix} {\rm ,}\\[1ex]
\differential^{-1}(x) = & \begin{pmatrix}
 x^{n-1} \\
 n(n-1)x^{n-1} - (n-1)(n-2)x^{n-2}
\end{pmatrix} {\rm ,} \\[1ex]
\differential^{-1}(x^{\geq 2}) = & \begin{pmatrix}
 x^{n-1} \\
 2(n-1)x^{n-1}
\end{pmatrix} {\rm .}
\end{align*}

To get into the spirit of things, let us abandon the general case for a second.
We refer the reader to \Cref{fig:dotdiagram}, which illustrates the complex $S$ described above in the case $n=5$.  The game we play is the following.  First we find a cocycle representative for $[\psi]$.  It shall turn out that the first representative we find generates the red dot in \Cref{fig:dotdiagram}.  Then we take a line (the gray line in that figure is one such example) of a slope between $0$ and $\infty$.  We position this line as far left as we can \emph{while still having some cocycle cohomologous with $\psi$ supported on or to the left of the line}.  Finally, we compute~$\upsln_5$: the slope of the line corresponds to some $t \in [0,1]$ and this leftmost position will determine~$\upsln_5(t)$.  In this particular case we find a cocycle cohomologous with $\psi$ supported in the green dots (although to verify this for herself, the reader will need to decorate the differentials in the figure with the correct coefficients).

We return now to the general case. Nevertheless, we will continue to rely on features of \Cref{fig:dotdiagram}.  Whenever we do this, the reader is invited to assure herself that the features relied upon do indeed hold for $3 \leq n \leq 10$.

We shall see later that the subcomplex $x^{n-1} C_{x^n - x^{n-1}}(D)$ has 1-dimensional cohomology, supported in cohomological grading $0$ and generated by $\psi$.  Therefore, we see immediately that $\psi$ is cohomologous to
$\psi_0 = (x^{n-1}, 0)$, since that vector lives in $x^{n-1}C^0_{x^n-x^{n-1}}(D)$, is a cocycle, and not a coboundary.

\Cref{fig:dotdiagram} shows $\langle\psi_0\rangle$ as a red dot.
It is clearly visible that any line passing through the red dot with slope between $0$ and $1/2$ has no other dots above it.
The red dot (a.k.a. $\psi_0$) has quantum filtration grading $-n-1$, $x$-filtration grading $n-1$,
and thus $\mathcal{G}_t$-grading ${-2t - n + 1}$ (see \cref{defn:slopingfiltration} and paragraphs thereafter).
So for $0 \leq t \leq 1/3$, one finds $\mathcal{G}_t^{-2t - n + 1} S = \langle \psi_0\rangle$
(note that $t = 1/3$ corresponds to a slope of $1/2$).
For those~$t$, $\psi_0$ is thus the `best' cocycle representative of $[\psi]$ with respect to the $\GG_t$-filtration.
Therefore, $\gamma(t) = -2t - n + 1$, and so $\upsln_n(t) = -n/(n-1)t$ for $t \in [0,1/3]$.

In fact, we claim that this holds even for $t\in [0,1/2]$.
To see this, let us verify that for $t \leq 1/2$, there is
no representative $\psi_1 \in [\psi] \cap  \GG_t^\ell S$ with $\ell < -2t - n + 1$. For such $t$ and $\ell$,
\begin{equation} \label{eq:space}\tag{$\dagger$}
\GG_t^\ell S \subset \langle 1, \ldots, x^{n-2} \rangle   \oplus \langle 1, \ldots, x^{n-4} \rangle.
\end{equation}
This can be easily seen graphically by considering which dots lie strictly above a line of slope $1$ through the red dot
(this line is drawn in gray in \cref{fig:dotdiagram}).
The difference $\psi_0 - \psi_1$ must be null-cohomologous, so equal to $\differential^{-1}(\alpha)$ for some $\alpha$;
because $\differential^{-1}(x^i) = \differential^{-1}(x^2)$ for $i \geq 2$,
we may assume $\alpha = \lambda_0  + \lambda_1 x + \lambda_2 x^2$. Note that the second coordinate of $\psi_0 - \psi_1$
is in $\langle 1, \ldots, x^{n-4} \rangle$. This implies that $\lambda_0 = 0$, since the coefficient of $x^{n-3}$ in
$\differential^{-1}(\alpha)$ equals $-\lambda_0(n-1)(n-2)$. Graphically, in the example $n = 5$ shown in \cref{fig:dotdiagram}, the dot at $(-6,0)$
is the only one mapping to $(-8,2)$. Similarly, one finds $\lambda_1 = 0$ by considering the coefficient of $x^{n-2}$,
and finally $\lambda_2 = 0$ by considering the coefficient $x^{n-1}$.
But this implies $\psi_0 = \psi_1$, contradicting~$\psi_1 \in \GG_t^{\ell}$.

\begin{figure}[t]
\centering
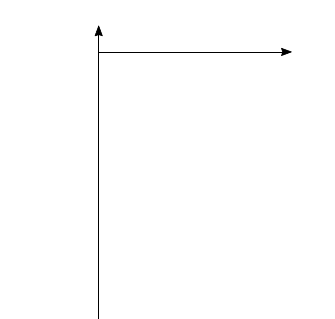%
\caption{$\upsln_5(K)$ for $K$ the $P(2,-3,7)$-pretzel knot.}
\label{fig:example_gimel}
\end{figure}

Now, let $\beta = 2(n-2) + 2nx - n^2x^2$ and compute
\begin{multline*}
\differential^{-1}(2(n-2) + 2nx - n^2x^2) =  \\
\begin{pmatrix}
 2n(n-2)x^{n-1} - 2(n-1)(n-2)x^{n-2} \\
 2n(n-1)(n-2)x^{n-2} - 2(n-1)(n-2)^2x^{n-3}
\end{pmatrix} \\
+
\begin{pmatrix}
 2nx^{n-1} \\
 2n^2(n-1)x^{n-1} - 2n(n-1)(n-2)x^{n-2}
\end{pmatrix}
+
\begin{pmatrix}
 -n^2x^{n-1} \\
 -2n^2(n-1)x^{n-1}
\end{pmatrix} \\
=
\begin{pmatrix}
n(n - 2)x^{n-1} - 2(n-1)(n-2)x^{n-2} \\
 - 2(n-1)(n-2)^2x^{n-3}
\end{pmatrix}.
\end{multline*}
This implies that $[\psi]$ has a representative $\psi_2 = \psi_0 - \differential^{-1}(\beta)/n(n-2)$ supported in $\langle x^{n-2}\rangle \oplus \langle x^{n-3} \rangle$
(support marked as green dots in \cref{fig:dotdiagram}). The $\GG_t$-filtration grading of $\psi_2$ is $-6t - n + 3$.
Hence $\upsln_n$ has a breakpoint at $t = 1/2$.
It is the only one, since for $\ell < -6t - n + 3$ and $t > 1/2$, we have once again \cref{eq:space}.
So, $\upsln_n(t) = -\frac{n+2}{n-1}t + \frac{1}{n-1}$ for $t \in [1/2,1]$. \Cref{fig:example_gimel} shows a plot of $\upsln_5(t)$.

\section{Proofs}
\label{sec:proofs}
In this paper we are mainly concerned with the potential $x^n - x^{n-1}$.  Nevertheless, in the first half of this section we shall work with a pair $(\potential, \alpha)$ given below.

\begin{definition}
	\label{defn:potential_and_simple_root}
	Let $\potential \in \CC[x]$ be a degree $n$ monic polynomial (the \emph{potential}), together with a root $\alpha$ of $\potential$ which occurs with multiplicity $1$.
\end{definition}

The point of enlarging our attention in this way is to arrive at \cref{defn:simple_root_bound} and \cref{prop:simple_root_bound} which give $s_{\potential, \alpha}(K) \in \Q$ whose absolute value gives a lower bound on the slice genus of $K$.  We expect this invariant to depend heavily on the choice of the pair $(\potential, \alpha)$ and to have properties analogous to those of the unreduced slice genus bounds discussed in \cite{lewarklobb}.  We do not, for example, expect it to give a knot concordance homomorphism unless one takes highly non-generic choices of the pair $(\potential, \alpha)$.  It is beyond the scope of this paper to explore these bounds further.  After the first half of this section we return to the potential $x^n - x^{n-1}$.

Another direction left unexplored in this paper is the construction of a $\upsln$-like concordance invariant in the case that the potential has the form $(x^n - x^{n-1})p$ for some monic $p \in \CC[x]$ with neither $0$ nor $1$ as roots.  This should \emph{a priori} be interesting for different choices of $p$ and of $n \geq 2$.  We note here too that the choice in this paper of potential $x^n - x^{n-1}$ is equivalent to the choice (\emph{mutatis mutandis}) of any potential $(x - \alpha)(x - \beta)^{n-1}$ with $\alpha, \beta \in \CC$, $\alpha \not= \beta$.

We begin with a definition of the Gornik cocycle $\psi$, which was originally defined by Gornik \cite{gornik} in the case that the potential is a product of distinct linear factors.

\begin{definition}
	\label{defn:psi}
	Suppose that $D$ is a link diagram.  The oriented resolution $O(D)$ of $D$ corresponds to a summand of the cochain group $C^0_{\potential}(D)$.  If $O(D)$ has $r$ components then this summand is isomorphic to
	\[ \CC [x_1, x_2,  \ldots, x_r]/(\potential(x_1), \potential(x_2), \ldots, \potential(x_r) ) \]
	where $x_i$ is the variable corresponding to the $i$-th component.  The special cocycle $\psi(D) \in C^0_{\potential}(D)$ is defined to be the element of this summand given by
	\[ \psi(D) = \prod_{i-1}^{r} \frac{\potential (x_i)}{x_i - \alpha} {\rm .} \]
\end{definition}

The following lemma is essentially due to Gornik.
\begin{lemma}
	\label{lem:psi_is_a_cocycle}
	We have that $\psi(D)$ is a cocycle. \hfill$\Box$ %
\end{lemma}

If $D$ is a basepointed diagram, then the cochain complex $C_{\potential}(D)$ has the structure of a complex of free $(\CC[x]/\potential)$-modules, where $x$ acts at the basepoint.  Note that $\psi(D)$ is an $\alpha$-eigenvector for the action of $x$, and so $[\psi(D)]$ represents a class in the $\alpha$-eigenspace of $H_{\potential}(D)$.  In fact, in the case that $D$ is a diagram of a knot, this $\alpha$-eigenspace is $1$-dimensional and supported in cohomological grading $0$.

A proof of this can be found in the proof of Theorem 2.15 of \cite{wedrich}, which considers colored perturbed $\SL_n$ cohomology of a $(1,1)$-tangle (which for us is the diagram $D$ cut open at the basepoint).  Specializing to the $1$-colored case and working with a general degree $n$ potential $\potential$, Theorem 2.15 identifies the cohomology of the $\lambda$-eigenspace of the complex $\CKR(D)$ with the $\SL_m$ cohomology of $D$ where $m$ is the multiplicity of $\lambda$ as a root of $\potential$.  Since the $\SL_1$ cohomology of a knot is $1$-dimensional, the result follows.

Now we know that either $[\psi(D)] = 0$ or $[\psi(D)]$ generates the $\alpha$-eigenspace.  To see that $[\psi(D)]$ is such a generator, one could generalize arguments of Gornik's.  We in fact deduce the result indirectly from the following proposition, whose proof in the separable potential case was given in \cite{lobb1} and \cite{wu3}, and extends to our current case with no changes.
\begin{proposition}
	\label{prop:1handle_psi_preserved}
	If $D_1$ is a link diagram obtained from the diagram $D_0$ by a $1$-handle attachment, then the induced cochain map takes $\psi(D_1)$ to a non-zero multiple of~$\psi(D_0)$.  \qed
\end{proposition}

\begin{remark}
	\label{rem:cohomology_contravariance}
	To be consistent with most of the literature, since we are speaking of knot \emph{cohomology} throughout this paper, we shall be thinking of the maps on the cohomology induced by link cobordism \emph{contravariantly}.  This is essentially a stylistic choice.
\end{remark}

With this \protect\Cref{prop:1handle_psi_preserved}  in hand we are now ready to deduce the following.

\begin{proposition}
	\label{prop:psi_is_not_a_coboundary}
	The class $[\psi(D)] \in H_{\potential}(D)$ is non-zero.
\end{proposition}
\begin{figure}[t]
\centering
\includegraphics{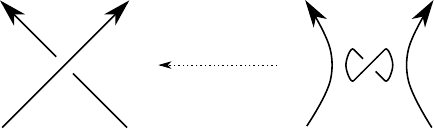}
\caption{Adding two $1$-handles, one on each side of a crossing.} 
\label{fig:trivialize}
\end{figure}
\begin{proof}
	By adding two 1-handles for every crossing of $D$ as shown in \protect\cref{fig:trivialize}, one obtains a presentation of a cobordism that takes a diagram $L$ of the unlink, in which every component has either zero or one crossings, to the diagram $D$.  In the light of \Cref{prop:1handle_psi_preserved}, it is therefore enough to verify that $\psi(L)$ represents a non-zero cohomology class.  Hence it is enough to verify that $\psi(U)$ represents a non-zero cohomology class when $U$ is a diagram of the unknot with at most one crossing.
	
	In the case of zero crossings or of a positive crossing this is trivially true since the coboundaries of cohomological grading $0$ consist of just the $0$ element.  In the case of $U$ with one negative crossing, the filtered degree of the differential ensures that $\psi(U)$ cannot be a coboundary.
\end{proof}

For any two basepointed diagrams of a knot $K$, there exists a sequence of basepoint-avoiding Reidemeister moves to get from one diagram to the other.  We note that the isomorphisms induced by such moves commute with the action of multiplication at the basepoint.  Therefore, since $\langle [ \psi(D) ] \rangle$ is 1-dimensional and characterized as the $\alpha$-eigenspace of such an action, we shall allow ourselves in future to refer to $\langle [\psi(D)] \rangle$ as $\langle [\psi(K)] \rangle$ when it makes sense to do so.

Suppose that $D$ is a diagram of a link with $k$ components, each component with a basepoint.  By acting at the $i$-th basepoint, we give $H_{\potential}(D)$ the structure of a $(\CC[x_i]/\potential(x_i))$-module.  This module structure is independent of the choice of the basepoint and of the diagram.

The class $[\psi(D)]$ is an $\alpha$-eigenvector for the action of each $x_i$.  In fact, it should be true that
\[ \langle [\psi(D)] \rangle = \bigcap_i \ker(x_i - \alpha : H_{\potential}(D) \rightarrow H_{\potential}(D)) {\rm .} \]
Unfortunately we have not been able to find a precise reference for this result, which deserves a more general treatment than we wish to give in this paper.  Consequently, we circumvent its use by appealing to a topological trick (stated and proved in more generality than we need in \cite{lobb1}).

\begin{lemma}
	\label{lem:nice_movie_presentation}
	Suppose we are given a cobordism $\Sigma \hookrightarrow S^3 \times [0,1]$ between two knots $K_i \hookrightarrow S^3 \times \{ i \}$ for $i = 0,1$, and a choice of diagram $D_i$ for each $K_i$.  Then there exists a movie presentation of $\Sigma$, starting with $D_0$ and ending with $D_1$, in which the elementary cobordisms occur in the following order.
	\begin{enumerate}
		\item Attachment of $0$-handles.
		\item Reidemeister moves.
		\item Attachment of $1$-handles.
		\item Reidemeister moves.
		\item Attachment of $2$-handles.
 \qed
	\end{enumerate}
\end{lemma}

The point of this topological trick for us is the following lemma.

\begin{lemma}
	\label{lem:1-eigenspace_for_special_links}
	Suppose that $D$ is a diagram of a link, where the link consists of a knot and the disjoint union of a $k$-component unlink.  Writing $x_i$ for $i = 0,1,\ldots ,k$ for basepoints on each component we have that
	\[ \langle [\psi(D)] \rangle = \bigcap_{i=0}^k \ker(x_i - \alpha : H_{\potential}(D) \rightarrow H_{\potential}(D)) {\rm .} \]
\end{lemma}
\begin{proof}
	First note that in the case $k=0$ the result holds.  In the case of $k \geq 1$, the result is then seen to be true when $D$ is the diagram in which each unknot component has $0$ crossings, since then the cochain complex is just a tensor product.  Finally note that the $\alpha$-eigenspace for the action of each $x_i$ does not change under Reidemeister moves.
\end{proof}

\begin{proposition}
	\label{prop:concordance_invariance}
	Suppose that $\Sigma$ is a smooth connected cobordism between knots $K_0$ and $K_1$, and suppose it has a given movie presentation between diagrams $D_0$ and $D_1$, satisfying the conditions of \protect\Cref{lem:nice_movie_presentation}.  Since all the $1$-handle attachments have been arranged to take place together, we assume they have been reordered so that the last $k$ $1$-handles each create a new component where $k$ is the total number of $2$-handles of the presentation.
	
	Then the map induced by the presentation
	\[ \Sigma_* : H_{\potential}^0(K_1) \rightarrow H_{\potential}^0(K_0) \]
	satisfies
	\[ \Sigma_* \langle [\psi(K_1)] \rangle = \langle [\psi(K_0)] \rangle \subseteq H_{\potential}^0 (K_0) { \rm .}\]
	Furthermore $\Sigma_*$ is induced by a cochain map of quantum filtration degree $(n-1)g(\Sigma)$ where we write $g$ for the genus.
	
\end{proposition}
\begin{proof}
	First note that $2$-handle attachment gives an element $1 \in H^0_{\potential}(U) = \CC[x]/\potential$ for $U$ the unknot.  This element of course has non-zero projection onto the $\alpha$-eigenspace of $x$ (using the projector which is multiplication by $\potential(x)/(x-\alpha)$).
	
	Let us then write the link appearing just before all the $2$-handle attachments as $K_{3/4} = K_1 \sqcup U_1 \sqcup \cdots \sqcup U_k$, and let us write $D_{3/4}$ for the diagram of this link that appears in the presentation just after all the $1$-handle attachments.  We start by considering the part of the presentation of $\Sigma$ that takes $D_{3/4}$ to $D_1$.
	
	Since the isomorphisms induced by Reidemeister moves commute with the action of the $x_i$, $[\psi(D_1)]$ is mapped to a non-zero element $\phi \in H^0_{\partial w}(K_{3/4})$ that lies in the $\alpha$-eigenspace of $x_0$ and that has non-zero projections onto the $\alpha$-eigenspace of each $x_i$ for $1 \leq i \leq k$.
	
	Now we consider the $1$-handle attachments.  The final $k$ of these split some knot $K_{1/2}$ with diagram $D_{1/2}$ into the link $K_{3/4}$ with diagram $D_{3/4}$.  Since $1$-handle attachments also commute with the action of the $x_i$, we know that $\phi$ must be mapped by the presentation of this cobordism into the $\alpha$-eigenspace of $H_{\potential}(K_{1/2})$.  That it gets mapped to a non-trivial multiple of the class $[\psi(D_{1/2})]$ follows from the observation that 
	\begin{align*} 0 &\not= \left( \prod_{i-1}^{r} \frac{\potential (x_i)}{x_i - \alpha} \right) \phi  \in \bigcap_{i=0}^k \ker(x_i - \alpha : H_{\potential}(K_{3/4}) \rightarrow H_{\potential}(K_{3/4}))  = \langle [\psi(K_{3/4})] \rangle {\rm ,}
	\end{align*}
	and \Cref{prop:1handle_psi_preserved}.
	
	We now write $D_{1/4} = D_0 \sqcup U'_1 \sqcup \cdots \sqcup U'_{\ell}$ for the diagram occurring just after the $0$-handle additions, and $K_{1/4}$ for the link that it represents.  It follows from \Cref{prop:1handle_psi_preserved} and the fact that the Reidemeister isomorphisms commute with the action at basepoints, that $[\psi(D_{1/2})]$ gets mapped to a non-trivial element of
	\[\bigcap_{i=0}^\ell \ker(x_i - \alpha : H_{\potential}(K_{1/4}) \rightarrow H_{\potential}(K_{1/4})) = \langle [\psi(D_{1/4})] \rangle {\rm .} \]
	It is then enough to observe that the map corresponding to $0$-handle addition on an empty diagram takes the element $[\psi(U)]$ to a non-zero scalar.
	
	The statement about the quantum filtration degrees is automatic from the definitions of the maps corresponding to the elementary cobordisms (see \cite{wu3} for explicit forms of these maps).
\end{proof}

\begin{definition}
	\label{defn:simple_root_bound}
	For a knot diagram $D$ let us write $u(D)$ for the quantum filtration grading of the class $[\psi(D)] \in H_{\potential}(D)$.
	Let $s_{\potential, \alpha}(D) \in \Q$ be given by
	\[ s_{\potential, \alpha}(D) = \frac{u(D) - n + 1}{2(n-1)} {\rm .} \]
\end{definition}

\begin{proposition}
	\label{prop:simple_root_bound}
	The quantity $s_{\potential, \alpha}(D)$ given in \cref{defn:simple_root_bound} only depends on the concordance class of the knot $K$ represented by $D$.  Furthermore, if we write $g_*(K)$ for the smooth $4$-ball genus of $K$ we have
	\[ g_*(K) \geq \vert s_{\potential, \alpha}(D) \vert {\rm .} \]
\end{proposition}

\begin{proof}
	This follows immediately from \cref{prop:concordance_invariance} and the observation that $u(U) = n-1$ for $U$ the $0$-crossing diagram of the unknot.
\end{proof}

We now return to the consideration of the case when the potential has the form $\potential = x^n - x^{n-1}$.

\begin{proposition}
	\label{prop:conc_invari_2}
	The map $\Sigma_*$ of \cref{prop:concordance_invariance} is induced by a cochain map of $x$-filtration degree $0$.
\end{proposition}
	
\begin{proof}
	Pick basepoints on $D_0$ and $D_1$ and connect these by a generic arc with no horizontal tangencies.  This gives a continuous choice of basepoints for each intermediate frame of the movie presentation (apart from those finite number of singular basepoints where the basepoint lies at a crossing).  Then the action of multiplication at intermediate basepoints either side of an elementary cobordism or singular basepoint commutes with the induced cochain map.
\end{proof}

With this in place we are now ready to begin deducing our main theorems as stated in \hyperref[sec:introduction]{the introduction}.

\begin{proof}[Proof of \Cref{thm:upsln_concordance_invariant}]
	 \Cref{prop:concordance_invariance} and \Cref{prop:conc_invari_2} imply that a concordance between knots $K_0$ and $K_1$ (that is a genus $0$ knot cobordism between $K_0$ and $K_1$) gives rise to a map $H_{\potential}(K_1) \rightarrow H_{\potential}(K_0)$ induced by a cochain map which preserves both the quantum and the $x$-filtrations.  Hence it preserves the $\GG_t$-filtration for all $t \in [0,1]$.  Furthermore, this map takes $\langle [ \psi(K_1) ] \rangle$ to $\langle [ \psi(K_0) ] \rangle$.  Likewise, by turning the concordance upside down we get a $\GG_t$-preserving map $H_{\potential}(K_0) \rightarrow H_{\potential}(K_1)$ taking $\langle [ \psi(K_0) ] \rangle$ to $\langle [ \psi(K_1) ] \rangle$.
	
	Since for any knot $K$ and any $t \in [0,1]$, the definition of $\upsln_n(K)(t)$ depends only on the $\GG_t$-grading of $\langle [ \psi (K) ] \rangle$, we have $\upsln_n(K_0) = \upsln_n(K_1)$.
\end{proof}

\Cref{thm:slicegenusbounds} will follow immediately from \Cref{prop:genus_cobordism_bounds}, whose proof we give now.

\begin{proof}[Proof of \Cref{prop:genus_cobordism_bounds}]
	Suppose that we have a genus $g$ cobordism $\Sigma$ from $K_0$ to~$K_1$.  We have seen that there exists a presentation of $\Sigma$ so that
	\[ \Sigma_* \langle [\psi (K_1)] \rangle = \langle [\psi(K_0)] \rangle {\rm .} \]
	Furthermore $\Sigma_*$ is induced by a cochain map which has degree $2(n-1)g$ with respect to the quantum filtration and degree $0$ with respect to the $x$-filtration.  Therefore, according to \Cref{defn:slopingfiltration}, $\Sigma_*$ has degree $2t(n-1)g$ with respect to the filtration $\GG_t$ for $t \in [0,1]$.
	
	We next turn to our definition of $\gamma$ from \Cref{defn:almostgimel}.  Since we know that the $\GG_t$-filtration is Reidemeister invariant, we can refer to $\gamma(K)$ for $K$ a knot.  There is a commutative square of maps in which the vertical arrows are induced by cobordism and the horizontal arrows by inclusion of complexes:
	\[\begin{CD}
		H(\GG_t^{\ell} \CKR (K_1)) @>>> \HKR(K_1) \\
		@VVV               @VVV \\
		H(\GG_t^{\ell+2t(n-1)g} \CKR (K_0)) @>>> \HKR(K_0) {\rm .}
	\end{CD}\]
	Hence we see that
	\[ \gamma(K_1)(t) + 2t(n-1)(g) \geq \gamma(K_0)(t) {\rm .}  \]
	Substituting this inequality into \Cref{defn:gimel} we see that
	\begin{align*}
	\upsln_n(K_1)(t) &= \frac{1}{2(n-1)} [\gamma(K_1)(t) - \gamma(U)(t)] \\
	&\geq \frac{1}{2(n-1)} [ \gamma(K_0)(t) - 2t(n-1)(g) - \gamma(U)(t) ] \\
	&= \upsln_n(K_0)(t) - tg {\rm ,}
	\end{align*}
	thus giving half of the inequality in \Cref{prop:genus_cobordism_bounds}. For the other half, consider turning $\Sigma$ upside down to get a cobordism from $K_1$ to $K_0$.  Then, arguing as before,
	\[ \pushQED{\qed}	\upsln_n(K_0)(t) \geq \upsln_n(K_1)(t) - tg {\rm .} \qedhere \popQED \]
\renewcommand{\qedsymbol}{}
\end{proof}

\begin{proof}[Proof of \Cref{thm:slicegenusbounds}] A slice surface of genus $g$ for a knot $K$ gives rise by puncturing the surface to a knot cobordism of genus $g$ between $K_0 : = K$ and $K_1 := U$, the unknot.  Now apply \Cref{prop:genus_cobordism_bounds}, and note that \Cref{defn:gimel} implies that $\upsln_n(U)$ is identically~$0$.
\end{proof}

We are now in a position where we can rapidly deduce that $\upsln_n$ is linear on some classes of knots.

\begin{proof}[Proof of \Cref{prop:boring_for_sqzd_knots}]
	The result is deduced immediately from \Cref{thm:slicegenusbounds} and from two facts.
	
	The first fact is that for any torus knot $T(p,q)$, $\upsln_n(T(p,q))$ is linear and of slope $-(p-1)(q-1)/2$ (note that the slope is, in absolute value, the slice genus of $T(p,q)$).  For positive torus knots this is straightforward, since the cochain complex is supported in non-negative gradings.  This means that the Gornik generator $\psi$ is cohomologous only to itself and so the statement can be deduced just at the cochain level.  For the negative case, consider the negative torus knot $T(p,-q)$ (where $p,q>0$) in the usual way as a diagram of a $p$-stranded braid closure.  Then note that, up to an overall shift in quantum grading, $\CKR^0(T(p,-q))$ and $d^{-1} (\CKR^{-1}(T(p,-q)))$ are isomorphic to $\CKR^0(T(p,-1))$ and $d^{-1} (\CKR^{-1}(T(p,-1)))$, respectively.  Since $T(p,-1)$ is a diagram of the unknot $U$, it follows that $\upsln_n (T(p,-q))$ differs from $\upsln_n(U)$ by an overall shift.  Since $\upsln_n(U)$ is the zero function, it follows that $\upsln_n(T(p,-q))$ is linear.  Finally, we know from \Cref{thm:derivativeatzero} that $\upsln_n(T(p,-q))'(0) = - \upsln_n(T(p,q))'(0)$.
	
	The second fact is that any knot $K$ which is either quasi-positive, quasi-negative, or homogeneous can be exhibited as a slice of a minimal-genus knot cobordism between a positive and negative torus knot (see \cite{lew2} for explicit constructions).
\end{proof}

We now turn to the proof of \cref{thm:derivativeatzero}.  We advise the reader to refresh her knowledge of the meaning of the function $\gamma$ as given in \cref{defn:almostgimel}, as well as of the graphical characterization of the $\GG_t$-filtration given just before the definition of~$\gamma$.

\begin{proof}[Proof of \cref{thm:derivativeatzero}]
	Let us write $x^{n-1} \CKR(D)$ for the subcomplex of elements of $x$-grading $n-1$ of the Khovanov-Rozansky complex of the based knot diagram $D$.  This inherits the quantum filtration by restriction.  We write
	\[ r := \min \{ \ell : [\psi] \in \im(H^0(\F^\ell x^{n-1} \CKR(D)) \rightarrow \HKR^0(D)) \} {\rm ,}\]
	for the \emph{reduced quantum grading} (here `reduced' refers to the `reduced' subcomplex $x^{n-1} \CKR(D)$).
		
	For small values of $t$, since the cochain complex is finitely generated, it follows that if the line of slope $\frac{t}{1-t}$ through $(r,n-1)$ intersects $(\ell,-\ell)$ then we have $\gamma(D)(t) = \ell$.
	
	Hence for small values of $t$ we have
	\[ \frac{n-1+\gamma(D)(t)}{r-\gamma(D)(t)} = \frac{t}{1-t} \]
	so
	\[ \gamma(D)(t) = rt - (1-t)(n-1) \]
	and hence we have
	\[ \gamma(D)'(0) = r + (n-1) {\rm .} \]
	
	Then, using the fact that for $U$ the unknot we have
	\[ \gamma(U) = (2t-1)(n-1) {\rm ,} \]
	we see that
	\[ \upsln_n(D)'(0) = r - (n-1) {\rm .} \]
	
	So it remains to show that this is a concordance homomorphism.
	In particular, we need to show that $\upsln_n(K_1 \hash K_2)'(0) = \upsln_n(K_1)'(0) + \upsln_n(K_2)'(0)$ for any pair of knots $K_1$ and $K_2$.
	
	We shall follow the second half of the proof of Theorem 2.8 in \cite{lewarklobb} and write $D_1$ and $D_2$ for two diagrams with marked points and $D = D_1 \hash D_2$ for the marked diagram formed by taking connect sum at the marked points.  We shall write $r_1$, $r_2$, and $r$ in the obvious way for the reduced quantum gradings.
	
	We write $\Phi$ for the map
	\[ \Phi : \CKR(D_1) \otimes \CKR(D_2) \rightarrow \CKR(D) \]
	induced by $1$-handle addition splitting $D$ into $D_1 \sqcup D_2$.  This map $\Phi$ restricts to map of subcomplexes
	\[ \widetilde{\Phi} : (x_1^{n-1} \CKR(D_1)) \otimes (x_2^{n-1} \CKR(D_2)) \rightarrow (x^{n-1} \CKR(D)) {\rm .} \]
	Following the argument in \cite{lewarklobb} (there replacing $\potential(x)$ by $x^n - x^{n-1}$ and $\alpha$ by $1$), we see that $\widetilde{\Phi}$ gives a filtered degree $n-1$ isomorphism of cochain complexes (with filtered degree $1-n$ inverse).  Hence we have
	\[ r_1 + r_2 - (n-1) = r \]
	so that
	\[ (r_1 - (n-1)) + (r_2 - (n-1)) = r - (n-1) \]
	as required.
\end{proof}

Next we prove two propositions that will allow us to conclude that we have quasi-additivity of $\upsln_n$.  The first says that the graph of $\upsln_n$ lies in the cone with apex the origin and two sides given by the slope at $0$ and the value at $1$.  The second is a boundedness result on the size of such a cone.

\begin{proposition}
	\label{prop:supportinacone}
	For any knot $K$ and $t \in [0,1]$ we have that
	\[ t\upsln_n(K)(1) \leq \upsln_n(K)(t) \leq t\upsln_n'(K)(0) {\rm .} \]
\end{proposition}

\begin{proposition}
	\label{prop:slope_vs_value_at_1}
	For any knot $K$ we have that
	\[ \upsln'_n(K)(0) - 1 \leq \upsln_n(K)(1) \leq \upsln_n'(K)(0) {\rm .} \]
\end{proposition}

To prove these we shall refer to the function $\gamma$ given in \cref{defn:almostgimel}, as well as the graphical characterization of the $\GG_t$-filtration given just before \cref{defn:almostgimel}.

\begin{proof}[Proof of \cref{prop:slope_vs_value_at_1}]
	Let us again write $x^{n-1} \CKR(D)$ for the subcomplex of elements of $x$-gradings $n-1$.  We consider the following two maps of complexes.
	\[ x^{n-1} \CKR(D) \hookrightarrow \CKR(D) \stackrel{x^{n-1}}{\longrightarrow} x^{n-1} \CKR(D) {\rm .} \]
	The first of these maps is inclusion and is filtered of quantum degree $0$, the second of these maps is filtered of quantum degree $2(n-1)$.  Hence, taking our definition of~$r$ from the proof of \cref{thm:derivativeatzero} and our definition
 of $u$ from \cref{defn:simple_root_bound} (dropping `$(D)$' to lighten notation),
	we must have
	\[ r - 2(n-1) \leq u \leq r {\rm .} \]
	Now we have already computed that $\gamma'(0) = r + (n-1)$, and by definition we have $\gamma(1) = u$, hence we have
	\[ \gamma'(D)(0) - 3(n-1) \leq \gamma(D)(1)  \leq \gamma'(D)(0) - (n-1) {\rm .} \]
	Finally we use the fact that for unknot $U$ we have $\gamma(U)(t) = (2t-1)(n-1)$ and \cref{defn:gimel} giving $\upsln_n$ in terms of $\gamma$.
	\begin{align*}\pushQED{\qed}
		\upsln_n'(D)(0) - 1 &= \frac{1}{2(n-1)} (\gamma'(D)(0) - \gamma'(U)(0)) - 1 \\
		&= \frac{1}{2(n-1)} (\gamma'(D)(0) - 4(n-1))\\
		&\leq \frac{1}{2(n-1)} (\gamma(D)(1) - (n-1)) =\frac{1}{2(n-1)} (\gamma(D)(1) - \gamma(U)(1)) \\
		&= \upsln_n(D)(1) \\
		&\leq \frac{1}{2(n-1)} (\gamma'(D)(0) - 2(n-1)) = \frac{1}{2(n-1)} (\gamma'(D)(0) - \gamma'(U)(0)) \\
		&= \upsln'_n(D)(0) {\rm .}
\qedhere\popQED
	\end{align*}
\renewcommand{\qedsymbol}{}
\end{proof}

\begin{proof}[Proof of \cref{prop:supportinacone}]
	Let $u$ and $r$ have the same meaning as in the proof directly above, and let us consider $\gamma(D)$.  We shall think of this following the graphical description given just before \cref{defn:almostgimel}.
	
	For small $t$, the line of slope $\frac{t}{1-t}$ through $(r,n-1)$ intersects the line given by $x+y=0$ in $(\gamma(D)(t), -\gamma(D)(t))$.  On the other hand, for values of $t$ close to~$1$, $(\gamma(D)(t), -\gamma(D)(t))$ lies on the line of slope $\frac{t}{1-t}$ through $(u,0)$.  One sees that $(r,n-1)$ is the first pivot point, and $(u,0)$ is the final pivot point.  In general, one computes $\gamma(D)(t)$ by finding the intersection of a line of slope $\frac{t}{1-t}$ with $x+y=0$, and as $t$ varies from $0$ to $1$, this line pivots on a finite number of integer points in $\Z \times (\Z \cap [0,n-1])$.
	
	Since we know the first pivot point and the final pivot point, it follows that for any $t$, $(\gamma(D)(t), -\gamma(D)(t))$ must lie on a line of slope $\frac{t}{1-t}$ which intersects both the straight line segment between $(u,0)$ and $(u,n-1)$, and the straight line segment between $(u,n-1)$ and $(r, n-1)$.  Suppose such a line runs through the point $(u,h)$ for $h \in [0,n-1]$.  Then, since the line must also run through the segment between $(u,n-1)$ and $(r, n-1)$, it follows that the slope $\frac{t}{1-t}$ must satisfy
	\[ \frac{t}{1-t} \geq \frac{n-1-h}{r-u} {\rm ,} \]
	which implies that
	\begin{align}
        \tag{$\ddagger$}
	\label{eqn:thing1}
	h &\geq n - 1 - \frac{t}{1-t}(r-u) {\rm .}
	\end{align}
	
	Now this line intersects $x+y=0$ at $(\gamma(D)(t), -\gamma(D)(t))$, hence
	\[ \frac{h+ \gamma(D)(t)}{u-\gamma(D)(t)} = \frac{t}{1-t} \]
	which implies that
	\[ \gamma(D)(t) = t(u+h) - h {\rm .} \]
	
	Hence we can compute $\upsln_n$ using \cref{defn:gimel}.
	\begin{align*}
		\upsln_n(D)(t) &= \frac{1}{2(n-1)}[t(u+h) - h - (2t-1)(n-1)] \\
		&= \frac{1}{2(n-1)}[h(t-1) + tu - (2t-1)(n-1)] {\rm .}
	\end{align*}
	
	Now we know that $h \leq n-1$ and we know from the proof of \cref{prop:slope_vs_value_at_1} that $u \geq r - 2(n-1)$, so we have
	\begin{align*}
		\upsln_n(D)(t) &\geq \frac{1}{2(n-1)}[(n-1)(t-1) + tu - (2t-1)(n-1)] \\
		&= \frac{t}{2(n-1)}[ u - (n-1). ]
	\end{align*}
	But we also know by definition that $u = \gamma(D)(1)$ so
	\begin{align*}
		\upsln_n(D)(1) &= \frac{1}{2(n-1)} [\gamma(D)(1) - \gamma(U)(1)] \\
		&= \frac{1}{2(n-1)} [u - (n-1)] {\rm .}
	\end{align*}
	Hence we deduce the first inequality of \cref{prop:supportinacone}.
	
	Now we use \cref{eqn:thing1} to conclude the remainder of the proposition.  We have
	\begin{align*}
		\upsln_n(D)(t) &= \frac{1}{2(n-1)}[h(t-1) + tu - (2t-1)(n-1)] \\
		&\leq \frac{1}{2(n-1)}[(n-1)(t-1) +t(r-u) + tu - (2t-1)(n-1)] \\
		&= \frac{t}{2(n-1)}[  r - (n-1)  ] {\rm .}
	\end{align*}
	But we computed in the proof of \cref{prop:slope_vs_value_at_1} that $\gamma'(0) = r + (n-1)$ so
	\begin{align*}
		\upsln'_n(D)(0) &= \frac{1}{2(n-1)}[ \gamma'(D)(0) - \gamma'(U)(0) ] \\
		&= \frac{1}{2(n-1)} [ r + (n-1) - 2(n-1) ] = \frac{1}{2(n-1)}[  r - (n-1)  ] {\rm ,}
	\end{align*}
	and hence we deduce the second inequality.
\end{proof}

Quasi-additivity of $\upsln_n$ now follows immediately.

\begin{proof}[Proof of \cref{thm:quasi-homomorphism}]
	\cref{prop:supportinacone} says that the graph of $\upsln_n(K)$ is supported in a cone determined by $\upsln_n'(K)(0)$ and $\upsln_n(K)(1)$.  Then \cref{prop:slope_vs_value_at_1} says that the cone can be taken to be the one given by the lines through the origin of slope $\upsln_n'(K)(0)$ and of slope $\upsln_n'(K)(0) - 1$.  But we know from \cref{thm:derivativeatzero} that $\upsln_n'(0)$ is a knot concordance homomorphism.
\end{proof}

Next we show that $\upsln_2$ contains exactly the same information as Rasmussen's $s$ invariant.  The essential point is that Rasmussen's invariant can be defined either from the average of two gradings in the unreduced Khovanov cohomology (as was done in \cite{ras}), or from a single grading in the reduced Khovanov cohomology.

\begin{proof}[Proof of \cref{prop:sl2isboring}]
	We fix now $n=2$, $\potential = x^2 - x$, and take our definitions of $r$ and $u$ from the proofs of the preceding propositions.
	
	The cohomology $\CKR(D)$ is then $2$-dimensional and supported in two quantum filtration gradings differing by 2, the average of which is $-s(D)$ (here the minus sign is introduced by a different convention in $\SL_2$ Khovanov-Rozansky cohomology compared to Khovanov cohomology).  The number $u$ is exactly the lower of these two filtration gradings so $u = -s(D)-1$.
	
	On the other hand, the cohomology of $x \CKR(D)$ is $1$-dimensional and supported in quantum filtration grading $-s(D) -1$, and this is exactly the number $r$.  Hence we have $r=u=-s(D)-1$.
	
	Now, considering the graphical definition of $\upsln_2(D)(t)$ in terms of lines of slope $\frac{t}{1-t}$, the family of lines pivots first at $t=0$ on the point $(r,1) = (-s(D)-1,1)$ and finally at $t=1$ on the point $(u,0) = (-s(D)-1,0)$.  Hence there cannot be any intermediate pivot points, from which we see that $\upsln_2(D)$ is linear, and the slope can be computed as $\frac{-s(D)}{2}$.
\end{proof}

Finally, we argue that $\upsln_n(1)$ is superadditive with respect to connect sum.

\begin{proof}[Proof of \cref{thm:superadditiveat1}]
	Let $D_1 \sqcup D_2$ be the disjoint union of the knot diagrams $D_1$ and $D_2$, and let $D$ be a knot diagram resulting from adding a $1$-handle connecting the two disjoint pieces.
	
	Now, following our notation in previous proofs, we write $u(D)$ for the quantum filtration grading of $[\psi(D)]$ in $\HKR(D)$ (and similarly for $D_1$, $D_2$, and $D_1 \sqcup D_2$).
	
	The $1$-handle addition induces a quantum filtered degree $n-1$ cochain map
	\[ h: \CKR(D) \rightarrow \CKR(D_1 \sqcup D_2) {\rm .}\]
	that, in cohomology, takes $[\psi(D)]$ to a non-zero multiple of $[\psi(D_1) \tensor \psi(D_2)]$.  It follows that
	\begin{align*}
		u(D_1) + u(D_2) &= u(D_1 \sqcup D_2) \\
		&\leq u(D) + n - 1.
	\end{align*}
	
	Now we have
	\begin{align*}\pushQED{\qed}
		\upsln_n(D_1)(1) + \upsln_n(D_2)(1) &= \frac{1}{2(n-1)} [u(D_1) - (n-1)] + \frac{1}{2(n-1)} [u(D_2) - (n-1)] \\
		&= \frac{1}{2(n-1)}[ (u(D_1) + u(D_2) - (n - 1)) - (n - 1)] \\
		&\leq \frac{1}{2(n-1)} [u(D) - (n-1)] = \upsln_n(D)(1).
\qedhere\popQED
	\end{align*}
\renewcommand{\qedsymbol}{}
\end{proof}

\section{Equivariant cohomology and concordance}
\label{sec:equi}
\newcommand{\rnmod}{\ensuremath{R_n}\textrm{-Mod}}
The goal of this section is to extract a smooth concordance invariant directly from the equivariant $\SL_n$ cochain complex of a knot.
It will unify all previously constructed concordance invariants coming from versions of $\SL_n$ cohomology.

First, let us give more details on equivariant cohomology, which was only briefly mentioned in \cref{sec:example}.
There is a version of $\SL_n$ cohomology for every monic polynomial $\potential$ of degree $n$.
Treating the coefficients of $\potential$ as formal variables yields the so-called \emph{equivariant cohomology} \cite{krasnerEquivariant}.
The equivariant $\SL_n$ complex associated to a basepointed link diagram $D$ is a finitely generated complex $C_{U(n)}(D)$ of free $R_n$-modules,
where $R_n$ is the graded $\mathbb{C}$-algebra
\[
R_n = \mathbb{C}[x, a_0, \ldots, a_{n-1}] / (a_0 + a_1x + \ldots + a_{n-1}x^{n-1} + x^n) \\
\]
with grading $\deg x = 2, \deg a_i = 2(n - i)$. Note that there is a graded isomorphism
$g_n: R_n \to \mathbb{C}[x, a_1, \ldots, a_{n-1}]$.
Two diagrams of the same basepointed link (in particular, two diagrams differing just by choice of basepoint on the same component)
have homotopy equivalent equivariant complexes.

One may evaluate at some $\xi \in \mathbb{C}^{n}$, i.e.\ apply the homomorphism $\ev_{\xi}: R_n \to \mathbb{C}[x]/\potential$
that sends $a_i \mapsto \xi_{i+1}$ for $i\in\{0,\ldots, n-1\}$. Applying $\ev_{\xi}$ to $C_{U(n)}(D)$
recovers the filtered complex $C_{\potential}(D)$ with $\potential = \xi_0 + \xi_1x + \ldots + x^n$.
It will also prove useful to evaluate partially at some $\xi \in \mathbb{C}^{n-1}$, i.e.\ apply the homomorphism $\ev'_{\xi}: R_n \to \mathbb{C}[x]$ that is the composition of $g_n$ and sending $a_i \mapsto \xi_i$ for $i\in\{1,\ldots,n-1\}$.
This yields a complex of free filtered $\mathbb{C}[x]$-modules.

Let us introduce some further notation. Denote by \rnmod{} the
category of finitely generated graded $R_n$-modules and grading-preserving homomorphisms.
Let $\mathcal{C}(\rnmod)$ be the category of finitely generated cochain complexes over $\rnmod$,
and $\mathcal{C}_f(\rnmod)$ its full subcategory of cochain complexes of shifted free modules
(a \emph{shifted free module} is a sum of copies of $R_n$ with various grading shifts).
Note that $C_{U(n)}(D) \in \mathcal{C}_f(\rnmod)$.
Krasner proved the homotopy type of $C_{U(n)}(D)$ to be invariant under Reidemeister moves.
However, inspection of his proof reveals that indeed the following form of invariance is shown.
\begin{lemma}\label{thm:acyclic_summands}
Let $D$ and $D'$ be two link diagrams of a link $L$.
Then there exist two acyclic cochain complexes $A, A' \in \mathcal{C}_f(\rnmod)$ such that
$C_{U(n)}(D) \oplus A \cong C_{U(n)}(D') \oplus A'$.\qed
\end{lemma}

We call an object of an additive category \emph{indecomposable} if it
is not isomorphic to the sum of two non-zero objects.
Let us now consider how equivariant Khovanov-Rozansky cohomology decomposes.
 
\begin{prop}\label{thm:equi}
Let $L$ be a link with a diagram $D$.
\begin{enumerate}[label=(\roman*)]
\item The equivariant Khovanov-Rozansky complex $C_{U(n)}(D)$ is isomorphic to a direct sum of an acyclic complex in $\mathcal{C}_f(\rnmod)$ and finitely many indecomposable non-acyclic complexes in $\mathcal{C}_f(\rnmod)$.
\item The isomorphism types of the non-acyclic summands do not depend on the choice of $D$, i.e.\ they are link invariants.
\item If $L$ is a knot, then there is precisely one non-acyclic summand with
Euler characteristic $1$.  All other summands have Euler characteristic $0$.
\end{enumerate}
\end{prop}
\begin{proof}
(i):
The category $\rnmod$ is abelian, and so $\mathcal{C}(\rnmod)$ is, too. Moreover $\mathcal{C}(\rnmod)$ is $\mathbb{C}$-linear (i.e.\ its Hom-spaces are $\CC$-vector spaces),
and has finite-dimensional Hom-spaces. By \cite{atiyah},
$\mathcal{C}(\rnmod)$ is thus Krull-Schmidt, meaning that its objects can be written in an essentially unique way as the sum
of finitely many indecomposable objects. So $C_{U(n)}(D)$ decomposes as a sum of complexes whose isomorphism types are uniquely determined.

Since the chain modules of $C_{U(n)}(D)$ are free, the chain modules of those summands are summands of free modules, which means they are projective modules.
But it is a well known theorem (not to be confused with the harder Quillen-Suslin theorem)
that graded projective modules over a graded polynomial ring are in fact graded free (see e.g.~\cite[section 6.13]{jacobson}).

(ii): This follows immediately from \cref{thm:acyclic_summands}.

(iii): Denote by $e$ the partial evaluation $\ev'_{(0,\ldots,0)}$ (in fact, evaluating at any other $\xi\in\mathbb{C}^{n-1}$ would work just as well).
Applying $e$ to $C_{U(n)}(D)$ yields a complex $e(C_{U(n)}(D))$ of free filtered $\mathbb{C}[x]$-modules.
Since $\mathbb{C}[x]$ is a PID, the indecomposable summands of $e(C_{U(n)}(D))$ are all isomorphic to either a rank-1 complex, or
a shift of $\mathbb{C}[x] \xrightarrow{x^k} \mathbb{C}[x]$ for $k\geq 0$. This is discussed in detail in \cite{khovanov-frobenius} for $n = 2$
and in \cite{krasnerEquivariant} for greater $n$. Since $H(e(C_{U(n)}(D)))$ has rank 1, there is exactly one summand of the first type.
All summands of the second type have Euler-characteristic 0.
Now let $X$ be an indecomposable summand of $C_{U(n)}(D)$. Then the cochain modules of $X$ have the same ranks as those of $e(X)$,
and $e(X)$ is a sum of some of the indecomposable summands of $e(C_{U(n)}(D))$.  This implies the statement.
\end{proof}
\begin{definition}
Let $K$ be a knot. We call the single summand of Euler characteristic~$1$ in the decomposition of \cref{thm:equi}
the \emph{equivariant Rasmussen invariant} and denote it by $S_n(K)$.
\end{definition}
As a consequence of \cref{thm:equi}, the equivariant Rasmussen invariant is a knot invariant, well-defined up to isomorphism.
All previously defined concordance invariants from $\mathfrak{sl}_n$ cohomology may be computed from $S_n(K)$, as the following
proposition shows.
\begin{prop}
	We have the following:
\begin{enumerate}[label=(\roman*)]
\item
There is a cocycle $\psi' \in x^{n-1} \ev_{(0,\ldots,0,-1)} S_n(K)$ that is not a coboundary.
Let $\psi$ be the Gornik cocycle for the potential $\potential = x^n - x^{n-1}$ (see \cref{defn:psi}).
Then $[\psi] \in H_{x^n-x^{n-1}}(D)$ and $[\psi'] \in H(\ev_{(0,\ldots,0,-1)} S_n(K))$ have
the same $\mathcal{G}_t$ grading for all $t$. This allows $\upsln_n$ to be computed from $S_n$.
\item Let $\potential = x^n + \xi_{n-1}x^{n-1} + \ldots + \xi_0$ be a polynomial with complex coefficients.
Suppose $\potential$ is a product of distinct linear factors and fix a root $\alpha$ of $\potential$.
Then $H_{\potential}(K)$ is isomorphic to $H(\ev_{(\xi_0, \ldots, \xi_{n-1})}(S_n(K)))$,
and $\widetilde{H}_{\potential, \alpha}(K)$ is isomorphic to the cohomology of
\[
\frac{\potential}{x - \alpha} \ev_{(\xi_0, \ldots, \xi_{n-1})}(S_n(K))[1-n].
\]
Both isomorphisms preserve the cohomological and the quantum grading.
\end{enumerate}
\end{prop}
\begin{proof}
(i): This follows immediately from the fact that the decomposition of $C_{U(n)}(D)$ into indecomposable summands respects both the quantum and
the $x$-filtration, and thus the $\mathcal{G}_t$-filtrations as well.

(ii): Taking cohomology of the decomposition of $C_{U(n)}(D)$ into indecomposable summands, one finds $H(\ev_{(\xi_0, \ldots, \xi_{n-1})}(S_n(K)))$
isomorphic to a subspace of $H_{x^n - x^{n-1}}(K)$. Because their dimensions agree, that subspace is actually the whole space.
The analogous argument may be applied to the reduced case.
\end{proof}
\begin{theorem} Let $K$ and $K'$ be knots.
\begin{enumerate}[label=(\roman*)]
\item A smooth connected cobordism $\Sigma$ of genus $g$ from $K'$ to $K$ with a fixed movie presentation induces a non-zero map $\Sigma_*: S_n(K) \to S_n(K')$ of degree $n g$.
\item If $g = 0$, then $\Sigma_*$ is a grading-preserving isomorphism.
\item $S_n(K \# K')$ is an indecomposable summand of $S_n(K) \otimes S_n(K')$.
\end{enumerate}
\end{theorem}
Note that $S_n$ gives a lower bound to the slice genus by (i), and is a smooth concordance invariant by (ii).
In particular, if $K$ is slice, then $S_n(K)$ is free of rank 1 without grading shifts.

\begin{proof}
Let $D$ and $D'$ be the diagrams of $K$ and $K'$, respectively occurring at the ends of the movie presentation of $\Sigma$.

(i): The movie presentation of $\Sigma$ is
a sequence of Reidemeister moves and handle attachments. To define $\Sigma_*$, one composes the maps associated to each of these moves.
To Reidemeister moves, associate the isomorphisms given in \cite{krasnerEquivariant}.
To 0-, 1-, and 2-handle attachments, associate the maps given by unit, saddle and trace (one may use tangles).
This is exactly the same way that $\Sigma$ induces a map of $C_{\potential}$, and so evaluation of $\Sigma_*$
gives the corresponding map $C_{\potential}(D) \to C_{\potential}(D')$.
Since that map is non-zero, so must be $\Sigma_*$.

(ii):
Denote by $\overline{\Sigma}$ the cobordism obtained by turning $\Sigma$ upside down, so that $\overline{\Sigma}_*: S_n(K') \to S_n(K)$.
Now $\overline{\Sigma}_* \circ \Sigma_*$ is a map $S_n(K) \to S_n(K)$, which is equal to $(\overline{\Sigma}\circ \Sigma)_*$.
For $k\geq 0$, the $k$-th power of this map is induced by the $k$-fold composition of $\overline{\Sigma}\circ \Sigma$ with itself,
so by (i), the $k$-th power is non-zero.
Since every endomorphism of an indecomposable module is either nilpotent or an isomorphism (\cite[Lemma~6]{atiyah}),
$\overline{\Sigma}_* \circ \Sigma_*$ is an isomorphism in $\mathcal{C}_f(\rnmod)$. The map $\Sigma_*\circ \overline{\Sigma}_*$ is an isomorphism by the same argument,
which implies that $\Sigma_*$ and $\overline{\Sigma}_*$ are isomorphisms.

(iii): Since $C_{U(n)}(D \# D') \cong C_{U(n)}(D) \otimes C_{U(n)}(D')$, one finds
$S_n(K) \otimes S_n(K')$ as a direct summand of $S_n(K\#K')$. It is not necessarily indecomposable,
but has Euler characteristic $1$, and so one of its indecomposable summands must have odd Euler characteristic.
By \cref{thm:equi}(iii), that summand must be $S_n(K\# K')$.

\end{proof}

\begin{remark}
\newcommand{\hun}{\check{C}_{U(n)}}
It is also of interest to consider the sum of all non-acyclic summands of $C_{U(n)}(D)$.
By the same argument as above, this is a link invariant (up to grading preserving isomorphism);
let us denote it by $\hun(D)$.
The graded dimension of this cochain complex equals the graded dimension of reduced $\SL_n$ cohomology.
So it follows from Rasmussen's work \cite{ras4} that the graded dimension of $\hun$ stabilizes
for large $n$, in a certain sense: the exponents of $q$ are of the form $a + bn$ with $a,b\in\Z$.
Experimentally, one observes a similar stabilization of the differentials of~$\hun$.
Indeed, this appears to be the case for all examples considered in this paper.
This observation supports the idea that there should exist an \emph{equivariant HOMFLYPT cohomology}
unifying all of the~$\hun$.
\end{remark}

\section{A second example}
\label{sec:outlook}
\begin{figure}[ht]
\centering
\includegraphics[scale=0.833]{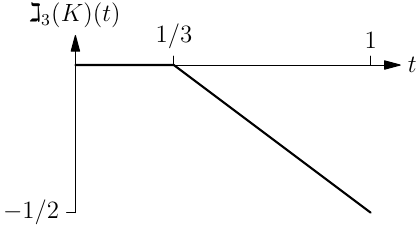}
\caption{$\upsln_3(K)$ for $K = P(7,-5,4) \# P(-9,7,-6)$.}
\label{fig:2xg}
\end{figure}
Let us calculate $\upsln_3$ of one more example knot $K = P(7,-5,4) \# P(-9,7,-6)$ (see \cref{fig:pretzel_example}).  We start by computing the equivariant Rasmussen invariant of $P(7,-5,4)$ and $P(-9,7,-6)$ using \khoca{} \cite{khoca}.
\[
\begin{array}{ll}
S_3(P(7,-5,4)): & 
S_3(P(-9,7,-6)): \\[2ex]
\xymatrix@R=1.5ex@C=4em{
                          & t^0 q^{0}R_n \ar@{}[dd]|{\oplus} & \\
t^{-1}q^{4}R_n \ar[ru]^{a_2^2 - 3a_1} \ar[rd]_{\potential'} \\
 & t^0q^{0}R_n
}
&
\xymatrix@R=1.5ex@C=4em{
t^0 q^0 R_n \ar@{}[dd]|{\oplus} \ar[rd]^{\ (3x + a_2)^3}\\
 & t^1 q^{-6} R_n  \\
t^0 q^{-2} R_n \ar[ru]_{\potential'}
}
\end{array}
\]
Taking the tensor product of those two complexes and plugging in $\potential = x^3 - x^2$ yields
the following complex over $\rho = \mathbb{C}[x]/(x^3 - x^2)$:
\[
\begin{array}{c}
t^{-1} q^4 \rho \\
\oplus \\
 t^{-1} q^2 \rho
\end{array}
\xrightarrow{\footnotesize\begin{pmatrix}
1          & 0          \\
3x^2 - 2x  & 0          \\
9x - 1     & 3x^2 - 2x  \\
0          & 1          \\
0          & 3x^2 - 2x  \\
\end{pmatrix}}
\begin{array}{c}
(t^{0} q^0\rho)^{\oplus 2} \\
\oplus\\
(t^{0} q^{-2}\rho)^{\oplus 3}
\end{array}
\xrightarrow{\footnotesize\begin{pmatrix}
9x - 1    & 0           \\
0         & 9x - 1      \\
-1        & -3x^2 + 2x  \\
3x^2 - 2x & 0           \\ 
0         & 3x^2 - 2x   \\
\end{pmatrix}^{\top}}
\begin{array}{c}
t^{1} q^{-6} \rho \\
\oplus \\
 t^{1} q^{-6} \rho
\end{array}
\]
One may take $\psi = (x^2, 0, 0, -8x^2, 0)^{\top}$, since this vector is in the kernel of $\differential^0$,
not in the image of $\differential^{-1}$, and in the highest $x$-filtration level.
For $t > 1/3$ and $\ell = t - 1$, one has
\[
\mathcal{G}^{\ell}_t C^0 = \langle 1, x\rangle^{\oplus 2} \oplus (q^{-2} \rho)^{\oplus 3}.
\]
That space contains a cochain cohomologous to $\psi$, namely
\[
\psi' = \differential^{-1}((-x^2 + x,0)^{\top}) + \psi = (x, 0, x^2 - x, -8x^2, 0)^{\top}.
\]
On the other hand, one easily checks that $\psi$ is not cohomologous to any cochain in 
$\langle 1\rangle^{\oplus 2} \oplus (q^{-2} \rho)^{\oplus 3}$. This implies that $\psi$ is
`best' for $t \leq 1/3$, and $\psi'$ is `best' for $t \geq 3$.
Hence $\gamma(K)(t) = 4t - 2$ for $t \leq 1/3$, and $\gamma(K)(t) = t - 1$ for $t \geq 1/3$,
and $\upsln_3(K)(t) = 0$ for $t \leq 1/3$ and $\upsln_3(K)(t) = -3t/4 + 1/4$.
A plot is shown in \cref{fig:2xg}.

This is an example of a knot for which all generalized Rasmussen invariants $s_2, s_3, \ldots$
vanish \cite{lew2}, but $\upsln_3$ obstructs its sliceness. Moreover, it is a quasi-alternating
knot \cite{greene2}, which implies that the concordance invariants coming from knot Floer homology
such as $\tau$ and $\Upsilon$ contain the same information as the knot signature $\sigma(K) = 0$.
This shows independence of $\upsln_3$ from these sliceness obstructions, and we deduce \Cref{prop:existence_of_knot_beating_upsilon_etcetera}.

With a little care, one can use $\upsln_3$ to show \Cref{prop:inf_order_in_QA_mod_A}.

\begin{proof}[Proof of \cref{prop:inf_order_in_QA_mod_A}]
	We begin the proof by assigning an exercise for the eager reader (she might start by taking the dual of $S_3(P(7,-5,4)) \otimes S_3(P(-9,7,-6))$ above).  The exercise is the computation of $\upsln_3(-K)$, where $-K$ is the concordance group inverse of $K$.  This will be found to be trivial.

The knot $K$ is of infinite order modulo $\langle A \rangle$ iff $K^{\# m}$ is not concordant to an alternating knot for any $m\geq 1$.
To prove this, it suffices to establish that $\upsln_3(K^{\# m})$ is non-linear for all $m \geq 1$.
Note that $\upsln_3(K)'(0) = 0$
implies
\begin{equation}\tag{$\dagger$}\label{eq1}
\upsln_3(K^{\# m})'(0) = 0
\end{equation}
for all $m \in \mathbb{Z}$, since $\upsln_3$ is a homomorphism near $0$.

Note moreover that $K$ is concordant to the following connected sum of $m$ knots:
\[
L = (K^{\# m}) \# \underbrace{-\!K \# \ldots \# -\!\!K}_{m-1}.
\]
So $\upsln_3(L)(1) = \upsln_3(K)(1) < 0$. By superadditivity at $1$ (\cref{thm:superadditiveat1}),
we have $\upsln_3(L)(1) \geq \upsln_3(K^{\# m})(1) + (m-1)\upsln_3(-K)(1) = \upsln_3(K^{\# m})(1)$,
and thus
\begin{equation}\tag{$\ddagger$}\label{eq2}
\upsln_3(K^{\# m})(1) < 0.
\end{equation}
Now \Cref{eq1} and \Cref{eq2} combined imply that $\upsln_3(K^{\# m})$ is non-linear.
\end{proof}

On the other hand, we note that $K$ has non-zero Levine-Tristram signature.  The Levine-Tristram signature of a knot $L$ is a function $\sigma_{\rm LT}(L) : [0,1] \rightarrow \Z$,
and is zero almost everywhere on algebraically slice knots (in other words those knots with Seifert matrices equal to that of a slice knot).

\begin{proposition}
	\label{prop:levine-tristram}
	There exists an algebraically slice knot $K'$ with $\upsln_3(K') \neq 0$
	but $s_n(K') = \tau(K') = 0$ and $\sigma_{\rm LT}(K') = 0$.
\end{proposition}
\begin{proof}
	First, choose two strongly quasipositive knots $P, Q$ with $g(P) = g(Q) = g(K)$, such that
	$P$ and $K$ have S-equivalent Seifert matrices, and $Q$ has Alexander polynomial~$1$ (that such choices exist is a consequence of \cite{Rudolph_83_ConstructionsOfQP1}).
	Then set $K' = K \# -P \# Q$.  This knot is clearly algebraically slice with $\sigma_{\rm LT}(K') = 0$. Moreover, since $P$ and $Q$ are quasipositive,
	they have equivariant Rasmussen invariants of rank 1, and so connected sum with $-P$ and $Q$
	merely results in two overall shifts for $\upsln_3$, which cancel each other.
\end{proof}

\bibliographystyle{myamsalpha}
\bibliography{References}
\end{document}